\documentclass{amsart}
\usepackage{amssymb}

\newtheorem{thm}{Theorem}[section]
\newtheorem{cor}[thm]{Corollary}
\newtheorem{lem}[thm]{Lemma}

\begin{document}
\title[On a generalization of the support problem of Erd\"{o}s ]{On a generalization of the support problem of Erd\"{o}s and its analogues for abelian varieties and $K$-theory}

\author{Stefan Bara\'{n}czuk} 
\address{Faculty of Mathematics and Computer Science, Adam Mickiewicz University\\ul. Umultowska 87, Pozna\'{n}, Poland}
\email{stefbar@amu.edu.pl}

\begin{abstract}In this paper we consider certain local-global principles for Mordell-Weil type groups over number fields like $S$-units, abelian varieties and  algebraic $K$-theory groups.
\end{abstract}



\maketitle

\section{Introduction.}

If $m$ is a positive integer then its support, denoted by $\mathrm{supp}(m)$, is the set of prime numbers dividing $m$ . 
In this paper we will prove 
\begin{thm}\label{uogolnienie erdosa}
Let $x_{1}, \ldots , x_{t}$ (resp. $y_{1}, \ldots , y_{t}$) be multiplicatively independent natural numbers  such that
\begin{displaymath}\bigcup_{i=1} ^{i=t} \mathrm{supp}(x_{i}^{n}-1)=\bigcup_{j=1} ^{j=t} \mathrm{supp}(y_{j}^{n}-1)
\end{displaymath}
for every natural number $n$. Then $\left\{x_{1}, \ldots , x_{t}\right\} = \left\{y_{1}, \ldots , y_{t}\right\}$.
\end{thm}
and its analogues for \textit{Mordell-Weil type groups} (see Corollary \ref{wniosek support}). Precise axiomatic setup for these groups can be found in \cite{Bar1}, below we just list those of them we consider in this paper:
\begin{enumerate}
	\item $\mathcal{O}_{F,S}^{\times}$, $S$-units groups, where $F$ is a number field and $S$ is a finite set of ideals in the ring of integers $\mathcal{O}_{F}$,
	\item $A(F)$, Mordell-Weil groups of abelian varieties over number fields $F$ with $\mathrm{End}_{\bar{F}} (A) = \mathbb{Z}$,
	\item $K_{2n+1}(F)$, $n>0$ , odd algebraic $K$-theory groups.
\end{enumerate}	
Theorem \ref{uogolnienie erdosa} generalizes  the \textit{support problem}, i.e. the following question of P\'{a}l Erd\"{o}s
\footnote{Generalizations in other directions can be found in \cite{S02}, \cite{Bar1} and \cite{Per}} :\\
\\
\textit{Let $x, y > 1$ be natural numbers such that
\begin{equation}
 \mathrm{supp}(x^{n}-1)=\mathrm{supp}(y^{n}-1) \label{warunek erdosa}
\end{equation}
for every natural number $n$. Is then $x=y$?}\\
\\
C. Corrales-Rodrig\'{a}\~{n}ez and R. Schoof answered the question affirmatively by proving the following theorem \footnote{Note that a more general result was obtained by A. Schinzel in \cite{S01}, see Theorem 2 } (and we will prove its generalization, see Theorem \ref{twierdzenie}):\\
\\
\textit{Let $F$ be a number field and let $x, y \in F^*$. If for almost all prime ideals $p$ of the ring of integers of $F$ and for all positive integers $n$ one has
\begin{equation}
y^n \equiv 1 (\mathrm{mod}\, p)\quad \mbox{whenever}\quad x^n \equiv 1 (\mathrm{mod}\, p) \label{warunek CorSch}
\end{equation}
then $y$ is a power of $x$.}\\
\\
They also proved an analogues theorem for elliptic curves and asked if it could be extended to abelian varieties. For special abelian varieties the problem was solved independently by G. Banaszak, W. Gajda and P. Kraso\'{n} in \cite{BGK2} and Ch. Khare and D. Prasad in \cite{KP}. The final solution of that problem was given by M. Larsen in \cite{Larsen}:\\
\\
\textit{
Let $F$ be a number field, $O_F$ its ring of integers, and $\mathcal{O}$ the coordinate ring of an open subscheme of $\mathrm{Spec} \, O_F$. Let $\mathcal{A}$ be an abelian scheme over $\mathcal{O}$ and $P,Q \in \mathcal{A}(\mathcal{O})$ arbitrary sections. Suppose that for all $n \in \mathbb{Z}$ and all prime ideals $p$ of $\mathcal{O}$, we have the implication\\
 \begin{displaymath}{if} \, \, \, nP \equiv 0 \, (\mathrm{mod}\, p) \, \, \, {then} \, \, \, nQ \equiv 0 \, (\mathrm{mod}\, p).
 \end{displaymath}
Then there exist a positive integer $k$ and an endomorphism $\phi \in End_{\mathcal{O}}(\mathcal{A})$ such
that $\phi(P) = kQ$.}\\ 
\\
In \cite{BGK1} also an analogues result for $K$- theory groups of number fields was proved.\\

Note that since ${\mathcal{O}_{F}/p}^{\times} $ is a cyclic group, the condition (\ref{warunek CorSch}) is equivalent to the following:
\begin{equation}
y \in \left\langle x\right\rangle \, (\mathrm{mod}\, p)\label{detecting}
\end{equation}
where $\left\langle x\right\rangle$ denotes the subgroup generated by $x$. However in the case of abelian varieties conditions (\ref{warunek CorSch}) and (\ref{detecting}) are not equivalent. It leads to the problem of \textit{detecting linear dependence by reduction maps} (see section \ref{section detecting}), formulated by W. Gajda in 2002 in a letter to Kenneth Ribet. The following theorem answering this question was proved by T. Weston in \cite{We}:\\
\\
\textit{Let $A$ be an abelian variety over a number field $F$ and assume that $End_FA$ is commutative. Let $\Lambda$ be a subgroup of $A(F)$ and suppose that $P \in A(F)$ is such that $P \in \Lambda \, (\mathrm{mod}\, v)$ for almost all places $v$ of $F$. Then $P \in \Lambda + A(F)_{\mathrm{tors}}$.}\\
\\
A similar result and its analogue for $K$-theory groups of number fields were proved independently by G. Banaszak, W. Gajda and P. Kraso\'{n} in \cite{BGK3}. Further results in abelian varieties case were obtained by W. Gajda and K. G\'{o}rnisiewicz in \cite{GG} and recently G. Banaszak proved the following theorem in \cite{Ban}:\\
\\
\textit{
Let $P_{1}, \ldots , P_{r}$ be elements of $A(F)$ linearly independent over $R=\mathrm{End}_{F} (A)$.
Let $P$ be a point of $A(F)$ such that $RP$ is a free $R$- module. The following conditions
are equivalent:\\
(1) \, $P \in \sum_{i=1}^{r} \mathbb{Z} P_{i}$\\
(2) \, $P \in \sum_{i=1}^{r} \mathbb{Z} P_{i} \, (\mathrm{mod} \, v)$ for almost all primes $v$ of $\mathcal{O}_{F}$.
}\\
\\
In Theorem \ref{detecting linear dependence} we generalize the problem of detecting linear dependence by reduction maps for Mordell-Weil type groups listed above
  analogously as we did it with  
the support problem in Theorem \ref{twierdzenie}.
\section*{Notation.}

\begin{tabular}{ll}
$B(F)$ &  a Mordell-Weil type group over a number field $F$ \\
$\Lambda_{\mathrm{tors}}$ &  the torsion part of a subgroup $\Lambda < B(F)$ \\
$\mathrm{ord} \ T$ & the order of a torsion point $T \in B(F)$ \\
$\mathrm{ord}_{v} P$ &  the order of a point $P \ (\mathrm{mod} \ v)$ where $P \in B(F)$ and $v$ is a prime of $F$ \\
										& (a prime of good reduction in abelian variety case and a prime not in $S$\\
										& in $S$-units case)\\
$l^{k} \parallel n$	& means $l^{k} \mid  n$		and	$l^{k+1} \nmid n$ where $l$ is a prime number, $k$ a nonnegative\\ 		& integer and $n$ a natural number.						
\end{tabular}\\

In the proofs we will use several times the following result which is a refinement of Theorem 3.1 of \cite{BGK3}:

\begin{lem}[\cite{Bar1}, Theorem 5.1]\label{o posylaniu}
Let $l$ be a prime number, $(k_{1}, \ldots, k_{m})$ a sequence of nonnegative integers. If $P_{1}, \ldots, P_{m} \in B(F)$ are linearly independent points, then there is a family of primes $v$ in $F$ such that $l^{k_{i}} \parallel \mathrm{ord}_{v} P_{i}$ if $k_{i}>0$ and  $l \nmid \mathrm{ord}_{v} P_{i}$ if $k_{i}=0$.
\end{lem}  

\section{The support problem.}

Fix one of the Mordell-Weil type group listed in the Introduction. In the following theorem all points belong to the fixed group and "linearly independent" (resp. "nontorsion") means linearly independent (resp. nontorsion) over $\mathbb{Z}$.

\begin{thm}\label{twierdzenie}
Let $Q_{1}, \ldots , Q_{t}$ be linearly independent points and $P$ be a nontorsion point.
Suppose that for almost all primes $v$ of $\mathcal{O}_{F}$ and all natural numbers $n$ the following condition holds :
\begin{eqnarray} if\  nP=0 \ (\mathrm{mod} \ v) \ then \ nQ_{i_{v}}=0 \ (\mathrm{mod} \ v)\ for \  some \ i_{v} \in \left\{1, \ldots , t\right\} .\label{warunek}
\end{eqnarray}
Then $Q_{i}=d P$ for some $i\in \left\{1, \ldots , t\right\}$ and some integer $d$. 
\end{thm}

\begin{proof}

\textbf{Step. 1.}
Suppose that points $P, Q_{1}, \ldots , Q_{t}$ are linearly independent. Fix arbitrary prime number $l$. By Lemma \ref{o posylaniu} there are infinitely many primes $v$ such that $l \nmid \mathrm{ord}_{v} P$ and $l \mid \mathrm{ord}_{v} Q_{i}$ for all $i$. Set $n=\mathrm{ord}_{v} P$. Then  
$nP=0 \ (\mathrm{mod} \ v)$ but $nQ_{i} \neq 0 \ (\mathrm{mod} \ v)$ for  all  $i \in \left\{1, \ldots , t\right\}$. Hence we get a contradiction by (\ref{warunek}) so $P, Q_{1}, \ldots , Q_{t}$ are linearly dependent.\\
\textbf{Step. 2.}
Now suppose that  there exist a nonzero integer $\alpha$ and nonzero integers $\beta_{1},\ldots,\beta_{k},\beta_{k+1}$ for some $k \geq 1$ such that $\alpha P = \beta_{1} Q_{1}+\ldots + \beta_{k} Q_{k} + \beta_{k+1}Q_{k+1}$ . Fix arbitrary  prime number $l$ coprime to $\alpha, \beta_{1},\ldots,\beta_{k},\beta_{k+1}$. Since $\beta_{1} Q_{1},\ldots , \beta_{k} Q_{k} , \beta_{1} Q_{1}+\ldots + \beta_{k} Q_{k} + \beta_{k+1}Q_{k+1}, Q_{k+2}, \ldots, Q_{t}$ are linearly independent then by Lemma \ref{o posylaniu} there are infinitely many primes $v$ such that 
\begin{center}
$l \parallel \mathrm{ord}_{v} \beta_{i} Q_{i}$ for $i \in \left\{1, \ldots , k-1\right\}\cup\left\{k+2, \ldots, t\right\}$,\\ 
$l^{2} \parallel \mathrm{ord}_{v} \beta_{k} Q_{k}$,\\
 $l \nmid \mathrm{ord}_{v} \beta_{1} Q_{1}+\ldots + \beta_{k} Q_{k} + \beta_{k+1}Q_{k+1}$.
\end{center} 
Let
\begin{center}
$n= \mathrm{lcm} (\frac{1}{l} \mathrm{ord}_{v} \beta_{1} Q_{1}, \ldots, \frac{1}{l} \mathrm{ord}_{v} \beta_{k-1} Q_{k-1}, \frac{1}{l^{2}} \mathrm{ord}_{v} \beta_{k} Q_{k}, \mathrm{ord}_{v} \beta_{1} Q_{1}+\ldots + \beta_{k} Q_{k} + \beta_{k+1}Q_{k+1},\frac{1}{l} \mathrm{ord}_{v} Q_{k+2}, \ldots, \frac{1}{l} \mathrm{ord}_{v} Q_{t})$.
\end{center}
Then $n(\beta_{1} Q_{1}+\ldots + \beta_{k} Q_{k} + \beta_{k+1}Q_{k+1})=0 \ (\mathrm{mod} \ v)$ but $n(\beta_{1} Q_{1}+\ldots + \beta_{k} Q_{k} )\neq 0 \ (\mathrm{mod} \ v)$ hence $n \beta_{k+1}Q_{k+1} \neq 0 \ (\mathrm{mod} \ v)$ but $l^{2}n \beta_{k+1}Q_{k+1} = 0 \ (\mathrm{mod} \ v)$ so $l \mid \mathrm{ord}_{v} \beta_{k+1} Q_{k+1}$. Thus $\alpha n P = 0 \ (\mathrm{mod} \ v)$ and by (\ref{warunek}) $\alpha n Q_{i}=0 \ (\mathrm{mod} \ v)$ for some $i$ so $l \mid \alpha$ . We get a contradiction so $k=0$.\\
\textbf{Step. 3.} 
By the previous step $\alpha P = \beta Q_{i}$ for some $i$ and nonzero integers $\alpha , \beta$.
Now let $l^{k} \parallel \beta$ for some prime number $l$ and positive integer $k$. By Lemma \ref{o posylaniu} there are infinitely many primes $v$ such that $l^{k} \parallel \mathrm{ord}_{v}  Q_{j}$ for $j \in \left\{1, \ldots , t\right\}$. So $\mathrm{ord}_{v}  Q_{i} = l^{k} m$ for some integer $m$ coprime to $l$ and $m \alpha  P= m \beta Q_{i} = 0 \ (\mathrm{mod} \ v)$ . Hence by (\ref{warunek}) $m \alpha  Q_{j}=0 \ (\mathrm{mod} \ v)$ for some $j$. Thus $l^{k} \mid \alpha$. Since $l$ was arbitrary we get $\beta \mid \alpha$, i.e. $d \beta P= \beta Q_{i}$ for some nonzero integer $d$.\\ 
\textbf{Step. 4.}
Now we repeat an argument from the proof of Theorem 3.12 of \cite{BGK3}:\\
Since $ \beta (d P -  Q_{i})=0$ then $d P - Q_{i}=T$ is a torsion point. Suppose that $T\neq 0$ and $l \mid \mathrm{ord}\ T$ for some prime number $l$. Again by Lemma \ref{o posylaniu} there are infinitely many primes $v$ such that $l \nmid \mathrm{ord}_{v} P$, $l \mid \mathrm{ord}_{v}Q_{j}$ for $j \neq i$. Thus by (\ref{warunek}) we get $l \nmid \mathrm{ord}_{v}Q_{i}$. Hence by definition of $T$ we have $l \nmid \mathrm{ord}_{v} T $. But for almost all primes $v$ $\mathrm{ord}\ T = \mathrm{ord}_{v} T$ by Lemma 3.11 of \cite{BGK3}. By a contradiction we get $T=0$.
\end{proof}

\begin{cor}\label{wniosek support} Let $P_{1}, \ldots , P_{t}$ (resp. $Q_{1}, \ldots , Q_{t}$) be linearly independent points such that for almost all primes $v$ of $\mathcal{O}_{F}$ and all natural numbers $n$
\begin{eqnarray} nP_{i}=0 \ (\mathrm{mod} \ v)\ for \  some \ i  \ \Leftrightarrow \ nQ_{j}=0 \ (\mathrm{mod} \ v)\ for \  some \ j  .\label{generalization}
\end{eqnarray}
Then there exist $\delta_{1}, \ldots, \delta_{t} \in \left\{-1, 1\right\}$ such that $\left\{P_{1}, \ldots , P_{t}\right\}=\left\{\delta_{1}Q_{1}, \ldots , \delta_{t}Q_{t}\right\}$.
\end{cor}

\begin{proof}
Applying Theorem \ref{twierdzenie} $2t$-times and using assumption of linear independence we get that for every point $P_{i}$ there is a unique point $Q_{j}$ such that $Q_{i}=k P_{j}$ and $P_{j}=l Q_{i}$ for some nonzero integers $k,l$. Hence $P_{j}=lk  P_{j}$ so $lk=1$. Thus $l,k \in \left\{-1, 1\right\}$.
\end{proof}

\begin{proof}[Proof of Theorem \ref{uogolnienie erdosa}] The statement comes immediately from Corollary \ref{wniosek support}, since if $x=y^{-1}$ and $x,y \in \mathbb{N}$ then $x=y=1$.
\end{proof}


\section{Detecting linear dependence by reduction maps.}\label{section detecting}

Let $B(F)$ denote one of the Mordell-Weil type group listed in the Introduction. 

\begin{thm}\label{detecting linear dependence}  Let $\Lambda$ be a subgroup of $B(F)$ and suppose that $P_{1}, \ldots , P_{n} \in B(F)$ are linearly independent points such that  for almost all primes $v$  
\begin{eqnarray} P_{i_{v}} \in \Lambda \, (\mathrm{mod}\, v) \, for \, some \,i_{v} \in \left\{1, \ldots, n\right\}  .\label{condition detecting}
\end{eqnarray}
Then $\alpha P_{i} \in \Lambda$ for some $i \in \left\{1, \ldots, n\right\}$ and $\alpha \in \mathbb{Z} \setminus \left\{0\right\}$.
\end{thm}

\begin{proof}
\textbf{Step. 1.}
Let $L_{1}, \ldots, L_{s}$ be a basis for the nontorsion part of $\Lambda$. Suppose that $P_{1}, \ldots, P_{n}, L_{1}, \ldots, L_{s}$ are linearly independent. Choose prime number $l$ such that $l$ does not divide the exponent of the group $\Lambda_{\mathrm{tors}}$. By Lemma \ref{o posylaniu} there are infinitely many primes $v$ such that 
\begin{center}
$l \mid  \mathrm{ord}_{v} P_{i}$ for $i \in \left\{1, \ldots , n\right\}$,\\ 
$l \nmid  \mathrm{ord}_{v} L_{j}$ for $j \in \left\{1, \ldots , s\right\}$.\\ 
\end{center}  
and by Lemma 3.11 of \cite{BGK3} for almost all primes $v$ we have $\mathrm{ord}_{v} T = \mathrm{ord} \ T$ for $T \in \Lambda_{\mathrm{tors}}$. Hence we get a contradiction with (\ref{condition detecting}) and $\alpha_{1}P_{1}+\ldots+\alpha_{n}P_{n}=\lambda_{1}L_{1}+\ldots+\lambda_{s}L_{s}$ with $\alpha_{1}, \ldots, \alpha_{n}, \lambda_{1}, \ldots , \lambda_{s} \in \mathbb{Z}$ where not all $\alpha_{i}$ and not all $\lambda_{j}$ are equal to $0$.\\
\textbf{Step. 2.} Suppose that the assertion of the theorem does not hold. Let us for simplicity of notation reorder the points $P_{1}, \ldots, P_{n}$ so that $\left\{L_{1}, \ldots, L_{s}, P_{1}, \ldots, P_{k}\right\}$ is a maximal subset of the set $\left\{L_{1}, \ldots, L_{s}, P_{1}, \ldots, P_{n}\right\}$ consisting of linearly independent points. Then for every $P_{w}$ where $w\in \left\{k+1, \ldots, n\right\}$ there are $\pi_{w}, \alpha_{1,w},\ldots,\alpha_{k,w}, \lambda_{1,w},\ldots,\lambda_{s,w} \in \mathbb{Z}$  such that
\begin{eqnarray}
\pi_{w}P_{w}= \alpha_{1,w}P_{1}+\ldots+\alpha_{k,w}P_{k}+\lambda_{1,w}L_{1}+\ldots+\lambda_{s,w}L_{s} \label{linear dependence step 2}
\end{eqnarray}
with $\pi_{w}\neq 0$ and $\alpha_{i,w}\neq 0$, $\lambda_{j,w}\neq 0$ for some $i,j$. Choose prime number $l$ coprime to all nonzero $\pi_{w}, \alpha_{1,w},\ldots,\alpha_{k,w}, \lambda_{1,w},\ldots,\lambda_{s,w}$ for all $w$ and to the exponent of the group $B(F)_{\mathrm{tors}}$. By Lemma \ref{o posylaniu} there are infinitely many primes $v$ such that 
\begin{center}
$l^{i+1} \parallel  \mathrm{ord}_{v} P_{i}$ for $i \in \left\{1, \ldots , k\right\}$,\\  
$l \parallel \mathrm{ord}_{v} L_{j}$ for $j \in \left\{1, \ldots , s\right\}$.\\ 
\end{center}
Let $i_{w}$ be the greatest number such that $\alpha_{i_{w},w}\neq 0$. 
Put  \begin{center}
$n= \mathrm{lcm} (\frac{1}{l^{2}} \mathrm{ord}_{v} P_{1}, \ldots, \frac{1}{l^{k}} \mathrm{ord}_{v} P_{k-1}, \frac{1}{l^{k+1}} \mathrm{ord}_{v} P_{k}, \frac{1}{l}\mathrm{ord}_{v} L_{1}, \ldots, \frac{1}{l}\mathrm{ord}_{v} L_{s})$.
\end{center}
Now we get by (\ref{linear dependence step 2}) that $n l^{i_{w}}\alpha_{i_{w}}P_{i_{w}}=n l^{i_{w}}\pi_{w}P_{w}$. Since $n l^{i_{w}}\alpha_{i_{w}}P_{i_{w}} \neq 0$, we get $n l^{i_{w}}\pi_{w}P_{w}\neq 0$. Since 
$n l^{i_{w}+1}\alpha_{i_{w}}P_{i_{w}} = 0$, we have $n l^{i_{w}+1}\pi_{w}P_{w} = 0$. Hence $l^{i_{w}+1} \parallel  \mathrm{ord}_{v} P_{w}$.
By (\ref{condition detecting}) we get a contradiction with orders of points $P_{1},\ldots,P_{k},P_{k+1},L_{1},\ldots,L_{s}$. So $k=0$. 
\end{proof}

\section*{Acknowledgments} 
The author would like to thank Grzegorz Banaszak and Wojciech Gajda for discussions, Krzysztof G\'{o}rnisiewicz for comments and Antonella Perucca for pointing out  mistakes in the first version of the paper.




\begin{thebibliography}{W}
\bibitem[Ban]{Ban} G. Banaszak, \textit{On a Hasse principle for Mordell-Weil groups}, preprint
\bibitem[Bar1]{Bar1} S. Bara\'{n}czuk, \textit{On reduction maps and support problem in $K$-theory and abelian varieties}, Journal of Number Theory \textbf{119} (2006), 1-17

\bibitem[BGK1]{BGK1} G. Banaszak, W. Gajda, P. Kraso\'{n}, \textit{A support problem for K-theory of number fields}, C. R. Acad. Paris S\'{e}r. 1 Math. \textbf{331 no. 3} (2000), 185-190 
\bibitem[BGK2]{BGK2} G. Banaszak, W. Gajda, P. Kraso\'{n}, \textit{Support problem for the intermediate jacobians of $l$-adic representations}, Journal of Number Theory \textbf{100} (2003) 133-168
\bibitem[BGK3]{BGK3} G. Banaszak, W. Gajda, P. Kraso\'{n}, \textit{Detecting linear dependence by reduction maps}, Journal of Number Theory 115 (2005) 322-342.

\bibitem[C-RS]{C-RS} C. Corrales-Rodrig\'{a}\~{n}ez, R. Schoof, \textit{The Support Problem and Its Elliptic Analogue}, Journal of Number Theory \textbf{64}, 276-290 (1997)

\bibitem[GG]{GG} W. Gajda, K. G\'{o}rnisiewicz, \textit{Linear dependence in Mordell-Weil groups}, to appear in Journal f\"{u}r die reine und angewandte Mathematik
\bibitem[KP]{KP} Ch. Khare, D. Prasad, \textit{Reduction of homomorphisms mod $p$ and algebraicity}, J. of Number Theory \textbf{105}, 322--332 (2004)
\bibitem[Lar]{Larsen} M. Larsen, \textit{The support problem for abelian varieties}, J. of Number Theory \textbf{101} (2003) 398–403
\bibitem[Per]{Per} A. Perucca, \textit{The l-adic support problem for abelian varieties},  arXiv:0712.2815 
\bibitem[Sch1]{S01} A. Schinzel, \textit{On power residues and exponential congruences}, Acta Arithmetica \textbf{27} (1975) 397-420 

\bibitem[Sch2]{S02} A. Schinzel, \textit{O pokazatelnych sravneniach}, Matematicheskie Zapiski \textbf{2} (1996) 121-126



\bibitem[We]{We} T. Weston , \textit{Kummer theory of abelian varieties and reductions of Mordell-Weil groups}, Acta Arithmetica \textbf{110.1} (2003), 77-88
\end{thebibliography}
\end{document}